\documentclass{amsart}
\usepackage{fullpage, amssymb, mathpazo, amsmath, hyperref, color, enumitem}
\usepackage{cite}
\renewcommand{\a}{\mathfrak a}
\newcommand{\cc}{\mathfrak c}

\newcommand{\Z}{\mathbb Z}

\newcommand{\Q}{\mathbb Q}
\newcommand{\R}{\mathbb R}
\newcommand{\C}{\mathbb C}

\newcommand{\bN}{\textbf{N}}

\renewcommand{\P}{\mathbb P}
\renewcommand{\to}{\rightarrow}
\newcommand{\longto}{\longrightarrow}
\newcommand*{\longinto}{\ensuremath{\lhook\joinrel\relbar\joinrel\rightarrow}}
\newcommand{\into}{\hookrightarrow}
\renewcommand{\o}{\mathcal O}
\newcommand{\calC}{\mathcal C}
\newcommand{\calI}{\mathcal I}
\newcommand{\calS}{\mathcal S}

\newcommand{\calU}{\mathcal U}
\newcommand{\calP}{\mathcal P}
\newcommand{\calN}{\mathcal N}
\newcommand{\calM}{\mathcal M}
\newcommand{\calX}{\mathcal X}

\renewcommand{\epsilon}{\varepsilon}
\newcommand{\Gal}{\operatorname{Gal}}
\newcommand{\Aut}{\operatorname{Aut}}
\newcommand{\cl}{\operatorname{Cl}}
\newcommand{\vol}{\operatorname{Vol}}

\newtheorem{thm}{Theorem}[section]
\newtheorem{lem}[thm]{Lemma}
\newtheorem{prop}[thm]{Proposition}
\newtheorem{cor}[thm]{Corollary}
\theoremstyle{definition}
\newtheorem{alg}{Algorithm}

\theoremstyle{remark}
\newtheorem{rem}[thm]{Remark}

\numberwithin{equation}{section}

\title{\small Computing points of bounded height in projective space over a number field}
\date{}

\author{David Krumm}
\address{Department of Mathematics \\
Claremont McKenna College \\
Claremont, CA 91711}
\email{dkrumm@cmc.edu}

\begin{document}
\begin{abstract} We construct an algorithm for solving the following problem: given a number field $K$, a positive integer $N$, and a positive real number $B$, determine all points in $\P^N(K)$ having relative height at most $B$. A theoretical analysis of the efficiency of the algorithm is provided, as well as sample computations showing how the algorithm performs in practice. Two variants of the method are described, and examples are given to compare their running times. In the case $N=1$ we compare our method to an earlier algorithm for enumerating elements of bounded height in number fields.
\end{abstract}
 
\maketitle

\section{Introduction} Let $K$ be a number field of degree $n$ over $\Q$ with ring of integers $\o_K$, and let $N$ be a positive integer. For any real number $B\ge 1$, define
\[\Omega(B)=\{P\in\P^N(K): H_K(P)\le B\},\]
where $H_K$ is the relative height function on the set $\P^N(K)$. In \cite{schanuel} Schanuel proved that there is a constant $c$, depending only on $N$ and on classical invariants of $K$, such that 
\[\#\Omega(B)\sim c\cdot B^{N+1}\;\;\; \text{as}\;\; B\to\infty.\] Thus, Schanuel's result provides a solution to the problem of estimating the number of points of bounded height in a projective space over $K$. In practice it can prove useful for various applications to have an algorithm corresponding to this counting problem, so that one can generate all points in $\Omega(B)$ for any given $B$. In the case $N=1$, algorithms of this type have been used to compute bases for Mordell-Weil groups of elliptic curves \cite{petho-schmitt}; to compute preperiodic points for quadratic polynomials \cite{jxd}; and to find examples of abelian surfaces with everywhere good reduction over quadratic fields \cite{dembele-kumar}. For larger values of $N$ an algorithm does not seem to exist in the literature; the purpose of this article is to develop an efficient algorithm that can be used for any value of $N$. One application of this more general algorithm would be to list points of small height on projective varieties. This can be useful, for instance, when computing rational points on curves. Let $C$ be a curve of genus $g\ge 2$ defined over $K$, and let $\mathcal J$ be its Jacobian variety. For the purpose of determining the set $C(K)$ of $K$-rational points on $C$ it is extremely useful to have a set of generators for the group $\mathcal J(K)$. A common way of producing generators is to search for points of small height on the associated Kummer variety $\mathcal K=\mathcal J/\{\pm 1\}$, which embeds in $\P^{2^g-1}$. By listing points of small height in $\P^{2^g-1}(K)$ one could therefore carry out an exhaustive search (within some bounds) for points on $\mathcal K$. Thus, one may hope to determine generators for the group $\mathcal J(K)$, and with additional work determine all points in $C(K)$. We refer the reader to \cite{stoll_height_const, stoll_height_const2, stoll_genus3}, where this technique is made very explicit for hyperelliptic curves of genera 2 and 3.
 
A natural first approach to the problem of computing $\Omega(B)$ is to turn Schanuel's counting argument into an algorithm. Unfortunately, it is not clear that this solution can be implemented in practice. Schanuel reduces the counting problem to a question of estimating the number of lattice points inside a certain region in $\R^{n(N+1)}$ (see Schanuel's paper for details, or \cite[Chap. 3, \S5]{lang_dg} for a sketch of the argument). While there are methods for computing lattice points inside bounded subsets of Euclidean space, these subsets must be relatively simple, and the dimension of the ambient space should be kept as small as possible. The region that occurs in Schanuel's paper is somewhat complicated (indeed, a substantial portion of the paper is spent on proving that the region is bounded and has a sufficiently smooth boundary), making it difficult to construct this region in a computer and determine all lattice points inside it. 

The method developed in the present article also involves a computation of lattice points in a bounded region; however, the ambient space can be taken to be either $\R^n$ or $\R^r$, where $r$ is the rank of the unit group of $\o_K$, and the bounded region is in both cases a polytope --- see Algorithms \ref{M1_alg} and \ref{M2_alg} below. The cost of this simplification of the problem is not large: if we measure the efficiency of the method by comparing the size of $\Omega(B)$ to the size of the search space $\calS(B)$ (that is, the set of all points generated by the algorithm while searching for points in $\Omega(B)$), then the inefficiency of the algorithm is bounded above by a constant as $B$ tends to infinity.

\begin{thm}\label{intro_thm}There is a constant $k$, depending on $N$, $K$, and a choice of fundamental units in $K$, such that 
\[\limsup_{B\to\infty}\frac{\#\calS(B)}{\#\Omega(B)}\le k.\]
\end{thm}

In the particular case $N=1$, where other methods already exist, our asymptotic bounds on the size of the search space compare favorably to those of the other methods: the algorithm of Peth\H{o} and Schmitt \cite{petho-schmitt} produces a search space that is larger than $\Omega(B)$ by a factor of $B^{2n-2}$; the algorithm given in \cite{doyle-krumm} improves this to a factor of $(\log B)^r$. For the algorithm developed here, this factor is a constant independent of $B$.

In addition to the efficiency of our algorithm there is another salient feature to point out. Given a height bound $B$, the algorithm computes a subset $\calC(B)\subset\o_K$ (defined in \S\ref{searchspace_section}) from which homogeneous coordinates of all points in $\Omega(B)$ can be taken. The set $\calC(B)$ depends on $B, K$, and a choice of fundamental units in $K$, but is independent of the dimension of the ambient space $\P^N$. Thus, by taking tuples of elements of $\calC(B)$ of appropriate length, one can determine all $K$-rational points of height bounded by $B$ in any projective space over $K$.

This article is organized as follows. In \S\ref{background_section} we set notation and record the theoretical results that are used in developing our algorithm. The main results of the paper are in \S\ref{bdd_class_section}, where we use the ideal class group of $K$ to partition the set $\Omega(B)$ into subsets, and show how to reduce the computation of these subsets to a problem of finding lattice points in polytopes. In \S\ref{searchspace_section} we discuss ways of improving the efficiency of our method. The special case where $K$ has only finitely many units is treated separately in \S\ref{rank0_section}. A proof of Theorem \ref{intro_thm} is given in \S\ref{efficiency_section}, and specific computations illustrating the performance of the algorithm are given in \S\ref{computations_section}.

\section{Background and notation}\label{background_section}

\subsection{Definition and computation of the relative height function}\label{notation_section}

Let $M_K$ denote the set of nontrivial places of $K$. To every place $v\in M_K$ there corresponds an absolute value $|\;|_v$ on $K$ extending one of the standard absolute values on $\Q$. We denote by $K_v$ the completion of $K$ with respect to $|\;|_v$, and by $\sigma_v$ the natural inclusion $K\into K_v$. Restricting $|\;|_v$ to $\Q$ and completing, we obtain a field $\Q_v$ which embeds into $K_v$ yielding a finite extension $K_v/\Q_v$; the \textit{local degree} of $K$ at $v$ is the degree $n_v$ of this extension. Defining a function $\|\;\|_v$ on $K$ by $\|x\|_v=|x|_v^{n_v}$, we have the following product formula for every $x\in K^{\ast}$:
\[\prod_{v\in M_K}\|x\|_v=1.\] 

For a point $P=[x_0,\ldots, x_N]\in\P^N(K)$, the {\it height} of $P$ (relative to $K$) is given by 
\[
H_K(P)= \prod_{v\in M_K}\max\{\|x_0\|_v,\ldots, \|x_N\|_v\}.
\]
The product formula ensures that the height of $P$ is independent of the choice of homogeneous coordinates for $P$. For purposes of explicit computations with absolute values and heights, a more concrete description of the places of $K$ and of the function $H_K$ will be needed.

Let $M_K^{\infty}$ denote the set of places $v$ for which $|\;|_v$ is Archimedean. For convenience, we will often write $v|\infty$ instead of $v\in M_K^{\infty}$. For every place $v|\infty$ the completion $K_v$ can be identified with either $\R$ or $\C$, so that the inclusion $\sigma_v:K\into K_v$ is identified with either a real or complex embedding of $K$. The local degree $n_v$ is 1 if $K_v=\R$ and 2 if $K_v=\C$. Hence, if the embeddings $K\into\C$ are explicitly known, one can use the relation $\|x\|_v=|\sigma_v(x)|^{n_v}_{\C}$ to compute, for any $x\in K$, all the numbers $\|x\|_v$ with $v|\infty$. Here, $|\;|_{\C}$ denotes the usual complex absolute value. This observation can be used to compute heights of points in $\P^N(K)$, as shown below.

We define a function $H_{\infty}:K^{N+1}\to\R$ by 

\[H_{\infty}(x_0,\ldots, x_N)=\prod_{v|\infty}\max\{\|x_0\|_v,\ldots, \|x_N\|_v\}.\]

Let $\sigma_1,\ldots, \sigma_{r_1}$ be the real embeddings of $K$, and $\tau_1,\overline\tau_1,\ldots, \tau_{r_2},\overline\tau_{r_2}$ the pairs of complex conjugate embeddings. It is a standard fact that the map $v\mapsto\sigma_v$ is a bijection between $M_K^{\infty}$ and the set $\{\sigma_1,\ldots, \sigma_{r_1},\tau_1,\ldots, \tau_{r_2}\}$. Using this fact we obtain the following alternate definition of $H_{\infty}$ which is more suitable for computation:
\begin{equation}\label{H_infty_comp}H_{\infty}(x_0,\ldots, x_N)=\prod_{\sigma}\max\{|\sigma(x_0)|_{\C},\ldots, |\sigma(x_N)|_{\C}\},
\end{equation}
where $\sigma$ ranges over all embeddings $K\into\C$.
It is well known that for any point $P=[x_0,\ldots, x_N]\in\P^N(K)$ we have the relation
\begin{equation}\label{rel_height_eq}
H_K(P)= H_{\infty}(x_0,\ldots, x_N)/N(\a),
\end{equation}
where $\a$ is the fractional ideal of $\o_K$ generated by $x_0,\ldots, x_N$. (For instance, see \cite[p. 136, 3.7]{silverman}.) Here, $N(\a)$ denotes the norm of the ideal $\a$. In practice, when computing $H_K(P)$ for a given point $P$, the formulas \eqref{rel_height_eq} and \eqref{H_infty_comp} will be used instead of our initial definition of the height.

\subsection{Results from Minkowski Theory}\label{mink_theory_section}

Let $n=[K:\Q]$ be the degree of $K$ over $\Q$. The {\it Minkowski embedding} of $K$ is the map
\begin{equation}\label{mink_embedding_eq}
\Phi:K\longinto\prod_{v|\infty}K_v\;\cong\;\R^n
\end{equation}
given by $x\mapsto (\sigma_v(x))_v$. The stated isomorphism of real vector spaces is not canonical, so we will make a choice: we fix the ordering $\sigma_1,\ldots, \sigma_{r_1},\tau_1,\ldots, \tau_{r_2}$ for the embeddings of $K$, thus obtaining an ordering $v_1,\ldots, v_{r_1+r_2}$ of the places $v|\infty$. This ordering induces an isomorphism
\[\prod_{v|\infty}K_v\;\cong\;\R^{r_1}\times\C^{r_2}.\]
Identifying $\C$ with $\R^2$ in the obvious way we obtain the isomorphism \eqref{mink_embedding_eq}, and in particular a more concrete description of the Minkowski embedding:
\[\Phi(x)=\left(\sigma_1(x),\ldots, \sigma_{r_1}(x),\Re\tau_1(x),\Im\tau_1(x),\ldots, \Re\tau_{r_2}(x),\Im\tau_{r_2}(x)\right).\]

Under the map $\Phi$, every nonzero ideal $\a$ of $\o_K$ becomes a lattice of rank $n$ in $\R^n$. More precisely, fixing an integral basis $\{\omega_1,\ldots,\omega_n\}$ for $\a$, the vectors $\Phi(\omega_1),\ldots, \Phi(\omega_n)$ are linearly independent over $\R$ and generate $\Phi(\a)$ as a $\Z$-module. We denote by $F(\a)$ the fundamental parallelotope spanned by these vectors:
\begin{equation}\label{ideal_fund_par}
F(\a)=\{c_1\Phi(\omega_1)+\cdots+c_n\Phi(\omega_n): c_i\in[0,1)\text{ for all $i$}\}.
\end{equation} 

The volume of $F(\a)$ is given by
\begin{equation}\label{ideal_volume}
\vol F(\a)=2^{-r_2}|\Delta_K|^{1/2}N(\a),
\end{equation}

where $\Delta_K$ is the discriminant of $K$; see \cite[p. 115, Lem. 2]{lang_nt}. Note that $F(\a)$ depends on the choice of integral basis for $\a$, but its volume does not.

Let $r=r_1+r_2-1$ be the rank of the unit group $\o_K^{\ast}$. Recall the standard logarithmic map 
\begin{equation}\label{log_map_eq}\Lambda: K^{\ast}\longto\prod_{v|\infty}\R\;\cong\;\R^{r+1}
\end{equation}
given by $x\mapsto(\log\|x\|_v)_v$. Again, the above isomorphism is not canonical, so we choose the isomorphism induced by the ordering $\sigma_1,\ldots, \sigma_{r_1},\tau_1,\ldots, \tau_{r_2}$. Thus, we obtain the more concrete description
\[\Lambda(x)=\left(\log|\sigma_1(x)|,\ldots, \log|\sigma_{r_1}(x)|,2\log|\tau_1(x)|,\ldots, 2\log|\tau_{r_2}(x)|\right).\]
A classical theorem of Dirichlet states that the image of $\o_K^{\ast}$ under the map $\Lambda$ is a lattice of rank $r$ in the hyperplane consisting of all points $(t_v)$ such that $\sum_v t_v=0$.
Any collection of units $\boldsymbol\epsilon=\{\epsilon_1,\ldots, \epsilon_r\}$ such that the vectors $\Lambda(\epsilon_j)$ form a basis for the lattice $\Lambda(\o_K^{\ast})$ is called a {\it system of fundamental units} in $K$. We fix a choice of fundamental units and define
\begin{equation}\label{fund_domain_units}
F(\boldsymbol\epsilon)= \{t_1\Lambda(\epsilon_1)+\cdots+t_r\Lambda(\epsilon_r): |t_j|\le 1/2 \text{ for all $j$}\}.
\end{equation}
Note that $F(\boldsymbol\epsilon)$ is the closure of a fundamental domain for the lattice $\Lambda(\o_K^{\ast})$. In each direction $v$ we will need to consider how far from the origin the vectors in $F(\boldsymbol\epsilon)$ can be; thus, we define numbers $D_v$ by
\begin{equation}\label{Dv_def_eq}
D_v= \max_{\eta\in F(\boldsymbol\epsilon)}\eta_v.
\end{equation}

We will also denote $D_v$ by $D_{\sigma}$ if $v$ corresponds to the embedding $\sigma\in\{\sigma_1,\ldots, \sigma_{r_1},\tau_1,\ldots, \tau_{r_2}\}$. Using the finite subset
\[V(\boldsymbol\epsilon)= \{t_1\Lambda(\epsilon_1)+\cdots+t_r\Lambda(\epsilon_r): |t_j|= 1/2\text{ for all $j$}\}\]
we obtain a different description of the numbers $D_v$ that is better for computation:
\begin{equation}\label{Dv_max_eq}
D_v= \max_{\eta\in V(\boldsymbol\epsilon)}\eta_v.
\end{equation}

\section{Points of bounded height with fixed ideal class}\label{bdd_class_section}

The starting point for our method is the observation that the set $\P^N(K)$ can be divided into ideal classes: to every point $P=[x_0,\ldots, x_N]\in\P^N(K)$ there corresponds the ideal class $\cl(P)$ of the fractional ideal generated by $x_0,\ldots, x_N$; this is independent of the choice of homogeneous coordinates for $P$. Since there are only finitely many ideal classes of $\o_K$, this observation reduces the problem of finding all points in $\Omega(B)$ to the following:

{\it Given a nonzero ideal $\a$ of $\o_K$, compute the set}
\[\Omega(\a,B)=\{P\in\P^N(K): \cl(P)=\cl(\a)\text{ and } H_K(P)\le B\}.\]

We will therefore begin by considering this specialized version of the main problem. We assume here that the unit rank $r$ is positive; the simpler case when $r=0$ is treated in \S\ref{rank0_section}.

\subsection{A search space for $\Omega(\a,B)$}\label{Omega_aB_section}
The main result of this section, namely Theorem \ref{main_alg_thm} below, will allow us to construct a finite set of points containing $\Omega(\a,B)$. Once this larger set is known, one can eliminate extraneous points from it by computing their heights and comparing to the bound $B$. 

We define a region
\[\calP(\a,B)\subset\prod_{v|\infty}K_v\;\cong\;\R^n\]
as follows: letting $D_v$ be the number \eqref{Dv_def_eq} for every $v|\infty$, the set $\calP(\a,B)$ consists of all points $(s_v)\in\prod_{v|\infty}K_v$ such that 
\[|s_v|\le\left(B\cdot N(\a)\right)^{1/n}\exp(D_v/n_v)\hspace{5mm} \forall\;v|\infty.\]

Note that as a subset of $\R^n$, $\calP(\a,B)$ is a Cartesian product of $r_1$ closed intervals in $\R$ and $r_2$ closed disks in $\R^2$. More precisely, $\calP(\a,B)$ consists of the points 
\[(a_1,\ldots, a_{r_1};x_1,y_1,\ldots, x_{r_2},y_{r_2})\in\R^n\] such that 
\[|a_i|\le\left(B\cdot N(\a)\right)^{1/n}\exp(D_{\sigma_i})\;\;\text{and}\;\;x_j^2+y_j^2\le\left(B\cdot N(\a)\right)^{2/n}\exp(D_{\tau_j})\;\;\]

for all indices $1\le i\le r_1$ and $1\le j\le r_2$.

\begin{thm}\label{main_alg_thm} For every point $P\in\Omega(\a,B)$ there exist $x_0,\ldots, x_N\in\a$ such that the following hold:
\begin{itemize}
\item $P=[x_0,\ldots, x_N]$;
\item $\a$ is the ideal generated by $x_0,\ldots, x_N$;
\item $|N_{K/\Q}(x_i)|\le B\cdot N(\a)$ for all $i$;
\item$\Phi(x_i)\in\calP(\a,B)$ for all $i$.
\end{itemize}
\end{thm}

In order to prove the theorem we will need some auxiliary results.

\begin{lem}\label{theta_map_lem} Define a map $\theta:K^{N+1}\setminus 0\longto\prod_{v|\infty}\R$ by 
\[\theta(\vec\alpha)=\left(\log\max_i\|\alpha_i\|_v\right)_v,\]
where $\vec\alpha=(\alpha_0,\ldots, \alpha_N)$. Then $\theta$ has the following properties:
\begin{itemize}
\item For every $\vec\alpha\in K^{N+1}\setminus 0$, $\sum_v\theta(\vec\alpha)_v=\log H_{\infty}(\vec\alpha)$.
\item If $u$ is a unit in $\o_K$, then $\theta(u\vec\alpha)=\Lambda(u)+\theta(\vec\alpha)$.
\end{itemize}
\end{lem}
\begin{proof}
Both properties follow immediately from the definitions.
\end{proof}

\begin{lem}\label{H_infty_scale_lem} For any $\lambda\in K$ and $Y\in K^{N+1}$ we have $H_{\infty}(\lambda Y)=|N_{K/\Q}(\lambda)|\cdot H_{\infty}(Y).$
\end{lem}
\begin{proof}
This is a consequence of \eqref{H_infty_comp}.
\end{proof}

\begin{proof}[Proof of Theorem \ref{main_alg_thm}] 
Let $P\in\Omega(\a,B)$. Since $\cl(P)=\cl(\a)$, there are homogeneous coordinates $[y_0,\ldots, y_N]$ for $P$ such that $\a$ is the ideal generated by $y_0,\ldots, y_N$. Letting $Y=(y_0,\ldots, y_N)\in K^{N+1}$, we have
\[H_{\infty}(Y)=H_K(P)\cdot N(\a)\le B\cdot N(\a).\]
The vectors $\Lambda(\epsilon_1),\ldots,\Lambda(\epsilon_r)$ together with the vector $(n_v)$ of local degrees form a basis for the Euclidean space $\prod_{v|\infty}\R$, so we can write
\begin{equation}\label{theta_Y_eq}
\theta(Y)=t\cdot (n_v)+\sum_{j=1}^rt_j\cdot\Lambda(\epsilon_j)
\end{equation}
for some real numbers $t,t_1,\ldots, t_r$. Considering the sum of the coordinates of the vectors on both sides of this equation, we find that 
\[t=\log H_{\infty}(Y)^{1/n}.\]
Here, we are using the first property of the map $\theta$ listed in Lemma \ref{theta_map_lem}. Let $n_j$ be an integer closest to $t_j$ for every $j$, so that $|t_j-n_j|\le 1/2$, and let $u=\epsilon_1^{-n_1}\cdots\epsilon_r^{-n_r}\in\o_K^{\ast}$. We now choose a different set of homogeneous coordinates for $P$: set $x_i=uy_i$ and $X=(x_0,\ldots, x_N)\in K^{N+1}$. Clearly, $P=[x_0,\ldots, x_N]$ and $\a$ is generated by $x_0,\ldots, x_N$; moreover, using Lemma \ref{H_infty_scale_lem} we see that $H_{\infty}(X)=H_{\infty}(uY)=H_{\infty}(Y)$, so our work above shows that
\[H_{\infty}(X)\le B\cdot N(\a)\;\;\;\text{and}\;\;\;t=\log H_{\infty}(X)^{1/n}.\] 
It follows from \eqref{H_infty_comp} that for every index $i$ we have $|N_{K/\Q}(x_i)|\le H_{\infty}(X)\le B\cdot N(\a)$.
It remains to show that $\Phi(x_i)\in\calP(\a,B)$ for all $i$.

Using the second property of $\theta$ stated in Lemma \ref{theta_map_lem} we obtain by \eqref{theta_Y_eq} that
\[\theta(X)=\theta(uY)=\Lambda(u)+\theta(Y)=t\cdot(n_v)+\eta\]
for some $\eta\in F(\boldsymbol\epsilon)$. Considering the equation $\theta(X)=t\cdot(n_v)+\eta$ one coordinate at a time we find that 
\[\log\max_i\|x_i\|_v=\log H_{\infty}(X)^{n_v/n}+\eta_v\hspace{5mm} \forall\;v|\infty,\]
and so
\[\max_i\|x_i\|_v=H_{\infty}(X)^{n_v/n}\exp(\eta_v)\hspace{5mm} \forall\;v|\infty.\]
Thus, for every index $i\in\{0,\ldots, N\}$ we have
\[\|x_i\|_v\le H_{\infty}(X)^{n_v/n}\exp(\eta_v)\le\left(B\cdot N(\a)\right)^{n_v/n}\exp(D_v)\hspace{5mm} \forall\;v|\infty.\]
By definition, this means that $\Phi(x_i)\in\calP(\a,B)$ for every $i$.
\end{proof}

Define a subset $\calC(\a,B)\subset\o_K$ by 
\[\calC(\a,B)=\{x\in\a: \Phi(x)\in \calP(\a,B) \text{ and } |N_{K/\Q}(x)|\le B\cdot N(\a)\}.\]
Note that $\calC(\a,B)$ is finite, since there can only be finitely many points of the lattice $\Phi(\a)$ lying inside the bounded region $\calP(\a,B)$. Theorem \ref{main_alg_thm} shows that every point $P\in\Omega(\a,B)$ has homogeneous coordinates coming from $\calC(\a,B)$, and thus provides a finite search space for the points in $\Omega(\a,B)$. More precisely, we have the following description of $\Omega(\a,B)$.

\begin{cor}
The set $\Omega(\a,B)$ consists of all points of the form $P=[x_0,\ldots, x_N]$ satisfying
\begin{itemize}
\item $x_i\in\calC(\a,B)$ for all $i$;
\item $\a$ is the ideal generated by $x_0,\ldots, x_N$; and 
\item $H_{\infty}(x_0,\ldots, x_N)\le B\cdot N(\a)$.
\end{itemize}
\end{cor}

\begin{proof}
Let $P\in\Omega(\a,B)$ be any point. By Theorem \ref{main_alg_thm} we know that there exist elements $x_0,\ldots, x_N\in\calC(\a,B)$ such that $P=[x_0,\ldots, x_N]$ and $\a$ is generated by $x_0,\ldots, x_N$. Moreover, since $H_K(P)\le B$, then \eqref{rel_height_eq} implies that $H_{\infty}(x_0,\ldots, x_N)\le B\cdot N(\a)$.

Conversely, suppose that $P=[x_0,\ldots, x_N]$ is any point whose coordinates satisfy the conditions listed in the corollary. The last two conditions then imply that $\cl(P)=\cl(\a)$ and $H_K(P)\le B$, so that $P\in\Omega(\a,B)$ by definition.
\end{proof}

 Using this corollary we obtain the following initial step towards our main algorithm.

\begin{alg}\label{omega_aB_alg}Computing $\Omega(\a,B)$.\mbox{}
\begin{enumerate}
\item Create an empty list $L$.
\item Compute the set $\calC(\a,B)$.
\item For every tuple $(x_0,\ldots, x_N)$ of elements of $\calC(\a,B)$:

\noindent If $\a$ is generated by $x_0,\ldots, x_N$, then:
\begin{enumerate}
\item Let $P=[x_0,\ldots, x_N]\in\P^N(K)$.
\item If $H_{\infty}(x_0,\ldots, x_N)\le B\cdot N(\a)$, then include $P$ in $L$.
\end{enumerate}
\item Return $L$.
\end{enumerate}
\end{alg}

\subsection{Exploiting group actions}\label{group_action_section}
Before considering how the steps of Algorithm \ref{omega_aB_alg} can be carried out in practice, we discuss a modification of the crucial step in the algorithm, which is to compute the set $\calC(\a,B)$. Once this set has been computed, the next step in Algorithm \ref{omega_aB_alg} is to build $(N+1)$-tuples of elements of $\calC(\a,B)$ and check the heights of the corresponding points in $\P^N(K)$. We will show here that it is possible to replace $\calC(\a,B)$ by a proper subset; this will have the effect of reducing significantly the number of tuples that need to be considered when computing $\Omega(\a,B)$. In addition, we show how to use two group actions to reduce the number of height computations that are carried out.

Let $\mu_K$ denote the group of roots of unity in $K$. From the definitions it follows that if $x\in\calC(\a,B)$ and $\zeta\in\mu_K$, then $\zeta x\in\calC(\a,B)$; hence, the group $\mu_K$ acts on $\calC(\a,B)$. Let $g_0,g_1,\ldots, g_t$ be elements representing all the orbits of this action. We will see that in Algorithm \ref{omega_aB_alg}, rather than considering all $(N+1)$-tuples of elements of $\calC(\a,B)$, one can restrict attention to tuples of elements of the set $\{g_0,g_1,\ldots, g_t\}$. Further reductions can be achieved by using two group actions on the set
\[K^{N+1}_{\bullet}=K^{N+1}\backslash\{(0,\ldots, 0)\}.\]
First, the action of the symmetric group $S_{N+1}$: for $\pi\in S_{N+1}$ and $X=(x_0,\ldots, x_N)\in K^{N+1}_{\bullet}$,
\[\pi\cdot X=(x_{\pi^{-1}(0)},\ldots, x_{\pi^{-1}(N)}).\]
Second, the action of the group $\mu_K^N$: for $u=(\zeta_0,\ldots, \zeta_{N-1})\in\mu_K^N$ and $X=(x_0,\ldots, x_N)\in K^{N+1}_{\bullet}$, 
\[u\cdot X=(\zeta_0x_0,\ldots, \zeta_{N-1}x_{N-1}, x_N).\]

For $X\in K^{N+1}_{\bullet}$ we let
\begin{equation}\label{orbit_def}
\o(X)=\{[u\cdot (\pi\cdot X)]: u\in\mu_K^N \;\text{and} \;\pi\in S_{N+1}\}\subset\P^N(K),
\end{equation}

where $[Y]$ denotes the equivalence class in $\P^N(K)$ of the point $Y\in K^{N+1}_{\bullet}$.

\begin{prop}\label{increasing_tuples_prop} Let $g_0,g_1,\ldots, g_t\in\calC(\a,B)$ be elements representing all the orbits of the action of $\mu_K$ on $\calC(\a,B)$. Let $\mathcal M$ be the set of all tuples $X=(g_{i_0},\ldots, g_{i_N})$ such that:
\begin{itemize}
\item $0\le i_0\le i_1\le\cdots\le i_N\le t$;
\item $\a$ is the ideal generated by $g_{i_0},\ldots, g_{i_N}$; and
\item $H_{\infty}(X)\le B\cdot N(\a)$.
\end{itemize}

Then
\[\Omega(\a,B)=\bigcup_{X\in\calM}\o(X).\]
\end{prop}

\begin{proof} The definition of $\calM$ implies that for any $X\in\calM$, the point $P=[X]$ satisfies $H_K(P)\le B$ and $\cl(P)=\cl(\a)$, so that $P\in\Omega(\a,B)$. A simple calculation shows that all the points in $\o(X)$ have the same height as $P$ and the same ideal class; therefore, $\o(X)\subseteq\Omega(\a,B)$. This proves one inclusion in the proposition. 

To see the reverse inclusion, let $P\in\Omega(\a,B)$. By Theorem \ref{main_alg_thm} we can write $P=[y_0,\ldots, y_N]$ with $y_i\in\calC(\a,B)$ generating the ideal $\a$. For every index $i$ we have $y_i=z_ig_{e_i}$ for some $z_i\in\mu_K$ and some index $e_i\in\{0,\ldots, t\}$; hence,
\[P=[z_0g_{e_0},\ldots, z_Ng_{e_N}].\]

Letting $\zeta_i=z_iz_N^{-1}$ for $i\in\{0,\ldots, N-1\}$ we obtain
\[P=[\zeta_0g_{e_0},\ldots, \zeta_{N-1}g_{e_{N-1}},g_{e_N}].\] 
By applying a permutation $\pi$ of the set $\{0,\ldots, N\}$ we may arrange the indices $e_i$ so that \[e_{\pi(0)}\le\cdots\le e_{\pi(N)}.\] Define  elements $x_0,\ldots,x_N$ by $x_i=g_{e_{\pi(i)}}$, so that $x_{\pi^{-1}(j)}=g_{e_j}$. Then
\[P=[\zeta_0x_{\pi^{-1}(0)},\ldots, \zeta_{N-1}x_{\pi^{-1}(N-1)}, x_{\pi^{-1}(N)}].\] 
Letting $X=(x_0,\ldots,x_N)\in K^{N+1}_{\bullet}$ and $u=(\zeta_0,\ldots, \zeta_{N-1})\in\mu_K^N$, this means by definition that
\[P=[u\cdot(\pi\cdot X)]\in\o(X).\]
We claim that $X\in\calM$. First of all, we have $X=(g_{e_{\pi(0)}}, \ldots, g_{e_{\pi(N)}})$ with increasing indices $e_{\pi(i)}$. Furthermore, from the construction of $X$ it follows that the entries $x_0,\ldots, x_N$ are associate -- in some order -- to $y_0,\ldots, y_N$. In particular, the ideal generated by $x_0,\ldots, x_N$ is equal to the ideal generated by $y_0,\ldots, y_N$, which is $\a$ by assumption. Finally, since all the points in $\o(X)$ have the same height, then
\[H_{\infty}(X)/N(\a)=H_K([X])=H_K(P)\le B,\] 
and hence $H_{\infty}(X)\le B\cdot N(\a)$. This shows that $X\in\mathcal M$, which completes the proof of the proposition.
\end{proof}

We deduce from Proposition \ref{increasing_tuples_prop} the following improvement of Algorithm \ref{omega_aB_alg} in which fewer points must be considered and fewer height computations are needed.

\begin{alg}\label{omega_aB_alg_improved}Computing $\Omega(\a,B)$.\mbox{}
\begin{enumerate}
\item Create an empty list $L$.
\item Compute representatives $g_0, g_1,\ldots, g_t$ of the orbit space $\calC(\a,B)/\mu_K$.
\item For every tuple of indices $(i_0,\ldots, i_N)$ such that $0\le i_0\le i_1\le\cdots\le i_N\le t$:

\begin{enumerate}
\item Let $X$ be the point $(g_{i_0},\ldots, g_{i_N})$.
\item If $\a$ is generated by $g_{i_0},\ldots, g_{i_N}$ and $H_{\infty}(X)\le B\cdot N(\a)$, then:

Include in $L$ all the points in $\o(X)$.
\end{enumerate}
\item Return $L$.
\end{enumerate}
\end{alg}

Precise details on how to carry out step (2) are given in Algorithm \ref{C_a_B_alg_second}.

\begin{rem} With the notation of Proposition \ref{increasing_tuples_prop}, it may very well happen that the sets $\o(X)$ and $\o(Y)$ for distinct tuples $X,Y\in\calM$ have a nonempty intersection. Furthermore, it may happen that the points of the form $[u\cdot(\pi\cdot X)]$ constructed when computing $\o(X)$ are not all distinct. Hence, in Algorithm \ref{omega_aB_alg_improved} one may want to remove duplicate points from the list $L$ before returning it. 
\end{rem} 

\begin{rem}\label{action_remark} Everything stated in this section remains valid if the group $\mu_K$ is replaced throughout by its subgroup $\{\pm 1\}$. This observation will be used below to give two different versions of our main algorithm.
\end{rem}

\subsection{Computing $\calC(\a,B)$}\label{C_aB_section}  We turn now to the question of how the set $\calC(\a,B)$ can be computed. Two approaches to this problem will be discussed, and in each case it will be shown that one can compute either $\calC(\a,B)/\mu_K$ or $\calC(\a,B)/\{\pm 1\}$ instead of the entire set $\calC(\a,B)$. By Proposition \ref{increasing_tuples_prop} and Remark \ref{action_remark}, having a complete set of representatives for either one of these orbit spaces is enough to determine $\Omega(\a,B)$ using Algorithm \ref{omega_aB_alg_improved}.

\subsubsection{First approach}\label{phi_approach}
The elements of $\calC(\a,B)$ can be found by first computing the larger set 
\[\widetilde\calC(\a,B)=\{x\in\a: \Phi(x)\in \calP(\a,B)\}\]
and then eliminating those elements of $\widetilde\calC(\a,B)$ whose norms exceed the bound $B\cdot N(\a)$. Note that the computation of $\widetilde\calC(\a,B)$ is equivalent to finding all points of the lattice $\Phi(\a)$ that lie inside the region $\calP(\a,B)$; it would therefore be desirable to make this region as small as possible. The size of $\calP(\a,B)$ is tied to the sizes of the numbers $D_v$ for $v|\infty$, which in turn are determined by the lengths of the vectors $\Lambda(\epsilon_1),\ldots, \Lambda(\epsilon_r)$ forming a basis for the lattice $\Lambda(\o_K^{\ast})\subset\R^{r+1}$. Thus, a basis consisting of short vectors should be computed. This can be achieved by first computing any system $\{\epsilon_1,\ldots, \epsilon_r\}$ of fundamental units (for instance, using the method described in \cite[\S 6.5.3]{cohen}) and then applying the LLL reduction algorithm \cite{lll}, or faster variants such as the Nguyen-Stehl\'e algorithm \cite{nguyen-stehle}, to the basis $\Lambda(\epsilon_1),\ldots, \Lambda(\epsilon_r)$. Having obtained a set of ``short" fundamental units in this way, the corresponding numbers $D_v$ can be computed using \eqref{Dv_max_eq}. The region $\calP(\a,B)$ is then determined, and we need to find all points of the lattice $\Phi(\a)$ that lie inside it. For this purpose it will be convenient to work with a slightly larger region defined as follows. Recall that $\calP(\a,B)$ consists of the points 
\[(a_1,\ldots, a_{r_1};x_1,y_1,\ldots, x_{r_2},y_{r_2})\in\R^n\] such that 
\[|a_i|\le\left(B\cdot N(\a)\right)^{1/n}\exp(D_{\sigma_i})\;\;\text{and}\;\;x_j^2+y_j^2\le\left(B\cdot N(\a)\right)^{2/n}\exp(D_{\tau_j})\;\;\]

for all indices $1\le i\le r_1$ and $1\le j\le r_2$. We define $\calP'(\a,B)\subset\R^n$ to be the region consisting of all points 
\[(a_1,\ldots, a_{r_1};x_1,y_1,\ldots, x_{r_2},y_{r_2})\in\R^n\] such that 
\[|a_i|\le\left(B\cdot N(\a)\right)^{1/n}\exp(D_{\sigma_i})\;\;\text{and}\;\;\max\{|x_j|,|y_j|\}\le\left(B\cdot N(\a)\right)^{1/n}\exp(D_{\tau_j}/2)\;\;\]

for all relevant indices $i,j$. Clearly $\calP'(\a,B)$ contains $\calP(\a,B)$, so in order to compute $\Phi(\a)\cap\calP(\a,B)$ it would suffice to compute $\Phi(\a)\cap\calP'(\a,B)$ and then eliminate points that do not satisfy the inequalities defining $\calP(\a,B)$. Note that $\calP(\a,B)$ is a Cartesian product of $r_1$ closed intervals in $\R$ and $r_2$ closed disks in $\R^2$. The definition of $\calP'(\a,B)$ differs from that of $\calP(\a,B)$ only in that each closed disk, say of radius $R$, is replaced by a square of side length $2R$ containing the disk. The ratio of the areas of the square and the disk is $4R^2/\pi R^2=4/\pi$. Thus,   
\[\frac{\vol \calP'(\a,B)}{\vol \calP(\a,B)}=\left(\frac{4}{\pi}\right)^{r_2}.\]
At the expense of increasing  the size of the region in which we search for lattice points (with a precise measure of the increase being given by the above equation), we gain the advantage of having a region that is easier to work with computationally. Indeed, $\calP'(\a,B)$ is a polytope -- i.e., the convex hull of a finite set of points -- and the problem of enumerating lattice points in polytopes has been well studied from a theoretical as well as computational point of view \cite{barvinok-pommersheim, barvinok, latte}.

In order to determine the points of the lattice $\Phi(\a)$ that lie inside the polytope $\calP'(\a,B)$, we translate this problem into one of finding \textit{integer} lattice points in a different polytope. Let $\{\omega_1,\ldots, \omega_n\}$ be an integral basis for the ideal $\a$, and let $S$ be the $n\times n$ matrix with column vectors $\Phi(\omega_1),\ldots, \Phi(\omega_n)$. The linear isomorphism $\R^n\to\R^n$ represented by the matrix $S^{-1}$ transforms the lattice $\Phi(\a)$ into the lattice $\Z^n$ and the polytope $\calP'(\a,B)$ into a polytope $\calX$. Note that for any integers $s_1,\ldots, s_n$,
\begin{equation}\label{polytope_trans}s_1\Phi(\omega_1)+\cdots+s_n\Phi(\omega_n)\in\calP'(\a,B) \iff (s_1,\ldots, s_n)\in \calX.\end{equation}
Hence, $\widetilde\calC(\a,B)$ is contained in the set of all numbers $x$ of the form $x=s_1\omega_1+\cdots+s_n\omega_n$ with $(s_1,\ldots, s_n)\in \Z^n\cap\calX$. We can therefore determine all elements of $\widetilde\calC(\a,B)$ once the integer points in $\calX$ are known: indeed, it suffices to construct all numbers $x$ of the above form and check the condition $\Phi(x)\in\calP(\a,B)$. An algorithm for computing integer lattice points in polytopes is described in \cite{latte} and implemented in the software package \texttt{LattE} \cite{latte_manual}. There is also ongoing work to include this algorithm in the \texttt{Sage} \cite{sage} software system.

To summarize this approach to computing $\calC(\a,B)$: first, a reduced basis for the lattice $\Lambda(\o_K^{\ast})$ is computed, and using it the polytope $\calP'(\a,B)$ is constructed. Computing an integral basis $\{\omega_1,\ldots, \omega_n\}$ of $\a$, an $n\times n$ matrix $S$ is defined to have columns $\Phi(\omega_1),\ldots, \Phi(\omega_n)$. The map $S^{-1}$ is then applied to $\calP'(\a,B)$ to obtain a new polytope $\calX$. Listing the integer points in $\calX$ we obtain a finite list of all integer tuples $(s_1,\ldots, s_n)$ with the property that the element $x=s_1\omega_1+\cdots+s_n\omega_n$ satisfies $\Phi(x)\in\calP'(\a,B)$. For all such elements $x$ we then check whether $\Phi(x)\in\calP(\a,B)$, thus obtaining the set $\widetilde\calC(\a,B)$. Finally, we compute the norms of all elements of $\widetilde\calC(\a,B)$ in order to check the inequality $|N_{K/\Q}(x)|\le B\cdot N(\a)$, and thus determine all elements of $\calC(\a,B)$.

As noted in the previous section, for the purpose of computing $\Omega(\a,B)$ it would suffice to find elements representing all the orbits in $\calC(\a,B)/\{\pm 1\}$. With minor modifications, the above procedure can be used to compute only these elements instead of all $\calC(\a,B)$. 

\begin{prop}\label{C_a_B_pm_reps} Let $\mathcal H\subset\R^n$ be the half-space consisting of all points whose first coordinate is non-negative. A complete set of representatives for the orbit space $\calC(\a,B)/\{\pm 1\}$ is given by all numbers of the form
\[x=s_1\omega_1+\cdots+s_n\omega_n\]
with $(s_1,\ldots, s_n)\in \Z^n\cap\mathcal H\cap\calX$ satisfying $\Phi(x)\in\calP(\a,B)$ and $|N_{K/\Q}(x)|\le B\cdot N(\a)$.
\end{prop}

\begin{proof}
Note to begin that, by definition, all elements $x$ of the above form indeed belong to $\calC(\a,B)$. We claim that they represent all the orbits of the action of $\{\pm 1\}$. Let $x$ be any element of $\calC(\a,B)$ and write 
\[x=s_1\omega_1+\cdots+s_n\omega_n\]
for some integers $s_1,\ldots, s_n$. Replacing $x$ with $-x$ if necessary, we may assume that $s_1\ge 0$. Since $\Phi(x)\in\calP'(\a,B)$, then \eqref{polytope_trans} implies that $(s_1,\ldots, s_n)\in\calX$ and hence $(s_1,\ldots, s_n)\in \Z^n\cap\mathcal H\cap\calX$. It follows that $x$ is one of the elements listed in the statement of the proposition.
\end{proof}

From Proposition \ref{C_a_B_pm_reps} and the preceding discussion we obtain the following algorithm.

\begin{alg}\label{C_a_B_alg_first}Computing $\calC(\a,B)/\{\pm 1\}$\mbox{}
\begin{enumerate}
\item Create an empty list $L$.
\item Compute an LLL-reduced system of fundamental units in $K$.
\item Compute the numbers $D_v$ for all places $v|\infty$.
\item Construct the polytope $\calP'(\a,B)$.
\item Compute an integral basis $\omega_1,\ldots, \omega_n$ for $\a$.
\item Let $S$ be the $n\times n$ matrix with column vectors $\Phi(\omega_i)$.
\item Construct the polytope $\calX=S^{-1}(\calP'(\a,B))$.
\item Find all integer lattice points in the polytope $\mathcal H\cap \calX$.
\item For all such points $(s_1,\ldots, s_n)$:
\begin{enumerate}
\item Let $x=s_1\omega_1+\cdots+s_n\omega_n$. 
\item If $\Phi(x)\in\calP(\a,B)$ and $|N_{K/\Q}(x)|\le B\cdot N(\a)$, then include $x$ in $L$.
\end{enumerate}
\item Return the list $L$.
\end{enumerate}
\end{alg}

\subsubsection{Second approach}\label{norm_approach} One possible issue with the method for computing $\calC(\a,B)$ discussed above is that in practice the set $\widetilde\calC(\a,B)$ can be significantly larger than $\calC(\a,B)$, so that the step of computing the norms of all elements of $\widetilde\calC(\a,B)$ is rather inefficient. We will therefore propose a different approach which reduces the number of norm computations needed. There is, however, a trade-off between the two approaches, since the second may require a substantial number of arithmetic operations with fundamental units. If $K$ is a number field for which these units are extremely large, it may be better to use the first method. 

In the second approach, rather than first finding all elements $x\in\a$ with $\Phi(x)\in\calP(\a,B)$ and then checking the condition $|N_{K/\Q}(x)|\le B\cdot N(\a)$, we first find elements $x$ satisfying this norm bound and then check whether $\Phi(x)\in\calP(\a,B)$. In general, there will be infinitely many elements $x\in\a$ with $|N_{K/\Q}(x)|\le B\cdot N(\a)$, since any such $x$ can be multiplied by a unit in $\o_K$ to obtain another element of $\a$ with equal norm. However, there are only finitely many possibilities for the ideal generated by $x$, since there are only finitely many ideals of bounded norm in $\o_K$. We will show here that the set $\calC(\a,B)$ can be determined by computing a finite list of principal ideals and a finite set of units. This approach is based on ideas first introduced in \cite{doyle-krumm}.

Let $\calI(\a,B)$ be a set of generators for all the nonzero principal ideals that are contained in $\a$ and whose norms are at most $B\cdot N(\a)$. We assume that distinct elements of $\calI(\a,B)$ generate distinct ideals. The elements of $\calI(\a,B)$ can be determined by using known methods for solving norm equations in number fields: applying the algorithm described in \cite{fincke-pohst1} (see also \cite{fincke-pohst2} and \cite[\S 5.3, \S 6.4]{pohst-z}) one can find generators for all principal ideals of $\o_K$ whose norms are of the form $k\cdot N(\a)$ with $1\le k\le B$. Keeping only those generators that belong to $\a$ we obtain $\calI(\a,B)$.

For every place $v|\infty$ we define real numbers $A_v$ and $L_v$ by the formulas
\begin{align*}
A_v &=\min_{y\in\calI(\a,B)}\Lambda(y)_v, \\
L_v &=(n_v/n)(\log B + \log N(\a))+D_v-A_v. 
\end{align*}

\begin{lem}\label{c_aB_xrep} Every nonzero element $x\in\calC(\a,B)$ can be written as $x=u\cdot y$, where $u\in \o_K^{\ast}$, $y\in\calI(\a,B)$, and
\begin{equation}\label{u_bounds}
\|u\|_v\le (B\cdot N(\a))^{n_v/n}\exp(D_v-A_v) \hspace{5mm} \forall\;v|\infty.
\end{equation}
\end{lem}

\begin{proof} Let $x\in\calC(\a,B)$ be any nonzero element. Since $|N_{K/\Q}(x)|\le B\cdot N(\a)$, the ideal generated by $x$ has norm at most $B\cdot N(\a)$, so $x$ must be associate to an element of $\calI(\a,B)$. Hence, we can write $x=uy$ for some unit $u$ and some $y\in\calI(\a,B)$. By definition, the fact that $\Phi(x)\in\calP(\a,B)$ means that $\|x\|_v\le (B\cdot N(\a))^{n_v/n}\exp(D_v)$ for every place $v|\infty$. Therefore, 
\[\|u\|_v=\|x\|_v\|y\|^{-1}_v\le (B\cdot N(\a))^{n_v/n}\exp(D_v)\exp(-A_v).\qedhere\]
\end{proof}

The above lemma shows that we can determine $\calC(\a,B)$ if we find all units $u$ satisfying the bounds \eqref{u_bounds}; to do this, we reduce the problem to one of finding integer points inside a polytope.  Let $\calU(\a,B)\subset\R^r$ be the polytope consisting of all points $(t_1,\ldots, t_r)$ that satisfy the inequalities $t_j  \le L_{v_j}$ for $1\le j\le r$, and $t_1+\cdots+t_r  \ge -L_{v_{r+1}}.$ Let $\pi:\R^{r+1}\to\R^r$ be the linear map that deletes the last coordinate, and let $\tilde\Lambda=\pi\circ\Lambda:K^{\ast}\to\R^r$. It is a standard fact that the image of the unit group $\o_K^{\ast}$ under $\tilde\Lambda$ is a lattice of full rank in $\R^r$ generated by the vectors $\tilde\Lambda(\epsilon_1),\ldots, \tilde\Lambda(\epsilon_r)$. Let $T$ be the $r\times r$ matrix having these vectors as columns. 

\begin{lem}\label{unit_polytope_lemma} 
For every unit $u\in\o_K^{\ast}$ satisfying the bounds \eqref{u_bounds} there exist a root of unity $\zeta\in\mu_K$ and an integer tuple $(n_1,\ldots, n_r)\in T^{-1}(\calU(\a,B))$ such that $u=\zeta\epsilon_1^{n_1}\cdots\epsilon_r^{n_r}$.
\end{lem}

\begin{proof} Let $t_1,\ldots, t_r$ be the coordinates of the vector $\tilde\Lambda(u)$. By definition we have $t_j=\log\|u\|_{v_j}$, so  \eqref{u_bounds} implies that $t_j\le L_{v_j}$. Moreover, the sum of the coordinates of $\Lambda(u)$ is 0, so 
\[t_1+\cdots+t_r=-\log\|u\|_{v_{r+1}}\geq -L_{v_{r+1}},\]
and hence $\tilde\Lambda(u)\in\calU(\a,B)$. Write $u=\zeta\epsilon_1^{n_1}\cdots\epsilon_r^{n_r}$ with $\zeta\in \mu_K$ and $n_1,\ldots,n_r\in\Z$. Then
\[n_1\tilde\Lambda(\epsilon_1)+\cdots+n_r\tilde\Lambda(\epsilon_r)=\tilde\Lambda(u)\in \calU(\a,B).\] 
Applying the linear map $T^{-1}$ we conclude that $(n_1,\ldots, n_r)\in T^{-1}(\calU(\a,B))$, and this proves the lemma.
\end{proof}

In view of Lemma \ref{unit_polytope_lemma}, the problem of computing $\calC(\a,B)$ is now reduced to that of  finding all integer points inside the polytope $T^{-1}(\calU(\a,B))$. As mentioned earlier, this kind of problem can be solved using the algorithm developed in \cite{latte}. 

Putting together our results in this section we obtain the following method for computing $\calC(\a,B)$:  first, the set $\calI(\a,B)$ is determined, and using this the numbers $A_v$ and $L_v$ are computed for every place $v|\infty$. (This requires previous knowledge of a system of fundamental units in $K$ from which the numbers $D_v$ are computed.) The polytope $T^{-1}(\calU(\a,B))$ is then constructed, and all integer lattice points inside it are found. For every such integer point $(n_1,\ldots, n_r)$, and for every root of unity $\zeta\in \mu_K$, we then construct all numbers of the form $x=\zeta\epsilon_1^{n_1}\cdots\epsilon_r^{n_r}y$ with $y\in\calI(\a,B)$. If $\Phi(x)\in\calP(\a,B)$, then we keep $x$ because it is an element of $\calC(\a,B)$; otherwise $x$ is discarded.

Using the above procedure we can compute the set $\calC(\a,B)$, and this could then be used to compute $\Omega(\a,B)$. However, to compute $\Omega(\a,B)$ using Algorithm \ref{omega_aB_alg_improved} it is enough to determine representatives for the orbit space $\calC(\a,B)/\mu_K$ instead of computing all of $\calC(\a,B)$. A small change to the method described above will allow us to compute only these representatives.

\begin{prop}\label{c_a_B_mod_mu_prop} Let $g_1,\ldots, g_t$ be all the numbers of the form $g=\epsilon_1^{n_1}\cdots\epsilon_r^{n_r}y$
with $y\in \calI(\a,B)$ and $(n_1,\ldots, n_r)\in \Z^r\cap T^{-1}(\calU(\a,B))$ satisfying $\Phi(g)\in\calP(\a,B)$. Then the numbers $0,g_1,\ldots, g_t$ form a complete set of representatives for the orbit space $\calC(\a,B)/\mu_K$.
\end{prop}

\begin{proof} Note first of all that, by construction, the numbers $0,g_1,\ldots, g_t$ all belong to $\calC(\a,B)$. Moreover, since no two elements of $\calI(\a,B)$ are associate, the orbits of $0,g_1,\ldots, g_t$ are all distinct. Let $x\in\calC(\a,B)$ be nonzero. By Lemmas \ref{c_aB_xrep} and \ref{unit_polytope_lemma} there exist $y\in\calI(\a,B)$ and $(n_1,\ldots, n_r)\in \Z^r\cap T^{-1}(\calU(\a,B))$ such that $x=\zeta\epsilon_1^{n_1}\cdots\epsilon_r^{n_r}y$ for some $\zeta\in\mu_K$. Thus, $x$ is in the $\mu_K$-orbit of the element $g=\epsilon_1^{n_1}\cdots\epsilon_r^{n_r}y$. Since $\mu_K$ acts on $\calC(\a,B)$, this implies that $g\in\calC(\a,B)$, so in particular $\Phi(g)\in\calP(\a,B)$. It follows that $g\in\{g_1,\ldots, g_t\}$, proving that $x$ is in the $\mu_K$-orbit of one of the numbers $g_1,\ldots, g_t$.
\end{proof}

We summarize the results of this section in the following algorithm.

\begin{alg}\label{C_a_B_alg_second}Computing $\calC(\a,B)/\mu_K$\mbox{}
\begin{enumerate}
\item Create a list $L$ containing only the element 0.
\item Compute $\calI(\a,B)$ by solving norm equations as described above.
\item Compute an LLL-reduced system $\{\epsilon_1,\ldots,\epsilon_r\}$ of fundamental units in $K$.
\item Compute the numbers $D_v$, $A_v$, and  $L_v$ for every place $v|\infty$.
\item Let $T$ be the $r\times r$ matrix with column vectors $\tilde\Lambda(\epsilon_1),\ldots, \tilde\Lambda(\epsilon_r)$.
\item Construct the polytope $T^{-1}(\calU(\a,B))$.
\item Find all integer lattice points inside $T^{-1}(\calU(\a,B))$.
\item For all such points $(n_1,\ldots,n_r)$, and for every element $y\in \calI(\a,B)$:
\begin{enumerate}
\item Let $g=\epsilon_1^{n_1}\cdots\epsilon_r^{n_r}y$.
\item If $\Phi(g)\in\calP(\a,B)$, then include $g$ in $L$.
\end{enumerate}
\item Return the list $L$.
\end{enumerate}
\end{alg}

Algorithms \ref{omega_aB_alg_improved}, \ref{C_a_B_alg_first}, and \ref{C_a_B_alg_second} provide two different ways of finding points of bounded height with specified ideal class. We turn now to the more general problem of determining all points of bounded height in $\P^N(K)$.

\section{A search space for all points of bounded height}\label{searchspace_section}
Given a real number $B\ge 1$, we wish to determine all points in the set
\[\Omega(B)=\{P\in\P^N(K): H_K(P)\le B\}.\]
If $\a_1,\ldots, \a_h$ are ideals representing the distinct ideal classes of $\o_K$, then we have
\[\Omega(B) = \bigcup_{i=1}^h\Omega(\a_i,B),\]
the union being disjoint. In order to compute $\Omega(B)$ it therefore suffices to determine ideal class representatives $\a_i$ as above and then compute $\Omega(\a_i,B)$ for every index $i$. The computational cost of obtaining the ideals $\a_i$ can be high if $K$ has very large discriminant; see \cite[Thm. 5.5]{lenstra} for a precise statement of the complexity of a deterministic algorithm. If one is willing to assume the Generalized Riemann Hypothesis, much faster methods are available: see, for instance, \cite{buchmann}. The main algorithms of this article include the computation of the class group as a required step; however, no assumptions are made as to which method is used for this. 

Once the ideal class representatives $\a_1,\ldots, \a_h$ have been determined, what remains in order to obtain $\Omega(B)$ is to compute $\Omega(\a_i,B)$ for every index $i$; this can be done by applying Algorithm \ref{omega_aB_alg_improved}. Though this approach to computing $\Omega(B)$ would certainly work, there are simple modifications that can be made to shorten the computation. The crucial step for finding all the points in $\Omega(\a_i,B)$ using Algorithm \ref{omega_aB_alg_improved} is to compute the set $\calC(\a_i,B)\subset\o_K$. Thus, in the process of determining all points in $\Omega(B)$ as described above, one would compute $\calC(\a_i,B)$ for every $i$. In practice there can be a significant amount of overlap between the various sets $\calC(\a_1,B), \ldots, \calC(\a_h,B)$,  so it can happen that the same elements of $\o_K$ are being computed several times. In order to avoid this redundancy, we will carry out one computation of a set $\calC(B)$ that contains all of the sets $\calC(\a_i,B)$, and then for each $i$ the elements of $\calC(\a_i,B)$ will be found by searching through $\calC(B)$. 

Let $\bN=\max_iN(\a_i)$ and let $\calP(B)$ be the subset of $\prod_{v|\infty}K_v\;\cong\;\R^n$ consisting of all points $(s_v)$ such that 
\[|s_v|\le\left(B\cdot \bN\right)^{1/n}\exp(D_v/n_v)\hspace{5mm} \forall\;v|\infty.\]
Note that $\calP(B)$ contains $\calP(\a_i, B)$ for every $i$. Define a set $\calN$ of non-negative integers by
\[\calN=\{0\}\cup\bigcup_{i=1}^h\{k\cdot N(\a_i):1\le k\le B\},\]
and let
\[\calC(B)=\{x\in\o_K: \Phi(x)\in \calP(B) \text{ and } |N_{K/\Q}(x)|\in\calN\},\]
so that $\calC(B)$ contains $\calC(\a_i, B)$ for every $i$. The methods of \S\ref{C_aB_section} can be easily adapted to compute $\calC(B)$. Once this has been done, the various sets $\calC(\a_i, B)$ can be determined by checking, for every element $x\in\calC(B)$, whether $x$ satisfies the conditions $\Phi(x)\in \calP(\a_i,B)$ and $|N_{K/\Q}(x)|\le B\cdot N(\a_i)$. The sets $\calC(\a_1,B), \ldots, \calC(\a_h,B)$ are thus obtained, and can then be used to determine all points in the sets $\Omega(\a_i,B)$. Now, in order to compute $\Omega(\a_i,B)$ using Algorithm \ref{omega_aB_alg_improved} it suffices to find representatives for the orbit spaces $\calC(\a_i,B)/\mu_K$ or $\calC(\a_i,B)/\{\pm 1\}$, so it would be desirable that group actions on $\calC(B)$ could be used to find these representatives instead of computing the entire set $\calC(\a_i, B)$. This can indeed be done without any additional work: the groups $\mu_K$ and $\{\pm 1\}$ act on $\calC(B)$, and with minor changes the methods of \S\ref{C_aB_section} can be used to compute representatives for the orbits of these actions; the details of this are discussed below. By selecting the elements that belong to $\calC(\a_i,B)$ for each $i$ we obtain representatives for the orbit spaces $\calC(\a_i,B)/\mu_K$ or $\calC(\a_i,B)/\{\pm 1\}$. These representatives can then be used in Algorithm \ref{omega_aB_alg_improved} to compute $\Omega(\a_i,B)$.

Depending on which orbit space is computed, $\calC(B)/\{\pm 1\}$ or $\calC(B)/\mu_K$, our discussion above yields a different method to 
compute $\Omega(B)$. We will henceforth denote by M1 the algorithm that uses the action of $\{\pm 1\}$, and by M2 the algorithm that uses the action of $\mu_K$. 

For the algorithm M1, the methods of \S\ref{phi_approach} should be applied to compute $\calC(B)/\{\pm 1\}$. Let $\calP'(B)\subset\R^n$ be the region consisting of all points 
\[(a_1,\ldots, a_{r_1};x_1,y_1,\ldots, x_{r_2},y_{r_2})\in\R^n\] such that 
\[|a_i|\le\left(B\cdot \bN\right)^{1/n}\exp(D_{\sigma_i})\;\;\text{and}\;\;\max\{|x_j|,|y_j|\}\le\left(B\cdot \bN\right)^{1/n}\exp(D_{\tau_j}/2)\;\;\]
for all indices $1\le i\le r_1$ and $1\le j\le r_2$.  Note that $\calP'(B)$ contains $\calP(B)$. Let $\{\omega_1,\ldots, \omega_n\}$ be an integral basis for $\o_K$, and let $S$ be the $n\times n$ matrix with column vectors $\Phi(\omega_1),\ldots, \Phi(\omega_n)$. Finally, let $\calX$ be the polytope $S^{-1}(\calP'(B))$. A minor modification of the proof of Proposition \ref{C_a_B_pm_reps} yields the following result.

\begin{prop}\label{C_B_pm_reps} Let $\mathcal H\subset\R^n$ be the half-space consisting of all points whose first coordinate is non-negative. A complete set of representatives for the orbit space $\calC(B)/\{\pm 1\}$ is given by all numbers of the form
\[x=s_1\omega_1+\cdots+s_n\omega_n\]
with $(s_1,\ldots, s_n)\in \Z^n\cap\mathcal H\cap\calX$ satisfying $\Phi(x)\in\calP(B)$ and $|N_{K/\Q}(x)|\in\calN$.
\end{prop}

The above proposition suggests an algorithm for computing $\calC(B)/\{\pm 1\}$ that is analogous (in fact, nearly identical) to Algorithm \ref{C_a_B_alg_first}. 

From our work up to this point we obtain the following description of M1. Briefly, what the algorithm below does is to compute $\calC(B)/\{\pm 1\}$, then intersect with each set $\calC(\a_i,B)$ to determine $\calC(\a_i,B)/\{\pm 1\}$, and finally use Algorithm \ref{omega_aB_alg_improved} to compute each set $\Omega(\a_i,B)$.

\begin{alg}[M1]\label{M1_alg} Computing $\Omega(B)$ using the action of $\{\pm 1\}$.\mbox{}
\begin{enumerate}
\item Create an empty list $\mathcal L$. This list will store the points belonging to $\Omega(B)$.
\item Compute an integral basis $\omega_1,\ldots, \omega_n$ for $\o_K$.
\item Determine ideals $\a_1,\ldots, \a_h$ representing the distinct ideal classes of $\o_K$.
\item Compute an LLL-reduced system of fundamental units in $\o_K$.
\item Construct the set $\calN$.
\item Compute the numbers $D_v$ for all places $v|\infty$.
\item Construct the polytope $\calP'(B)$.
\item Let $S$ be the $n\times n$ matrix with column vectors $\Phi(\omega_i)$.
\item Construct the polytope $\calX=S^{-1}(\calP'(B))$.
\item Create an empty set $L$. This set will store representatives for $\calC(B)/\{\pm 1\}$.
\item Find all integer lattice points in the polytope $\mathcal H\cap \calX$.
\item For all such points $(s_1,\ldots, s_n)$:
\begin{enumerate}
\item Let $x=s_1\omega_1+\cdots+s_n\omega_n$. 
\item If $\Phi(x)\in\calP(B)$ and $|N_{K/\Q}(x)|\in\calN$, then include $x$ in $L$.
\end{enumerate}
\item For each ideal $\a\in\{\a_1,\ldots, \a_h\}$:
\begin{enumerate}
\item Fix an ordering $g_0,g_1,\ldots, g_t$ of the elements of $L\cap\calC(\a, B)$.
\item For every tuple of indices $(i_0,\ldots, i_N)$ such that $0\le i_0\le i_1\le\cdots\le i_N\le t$:
\begin{enumerate}
\item Let $X$ be the point $(g_{i_0},\ldots, g_{i_N})$.
\item If $\a$ is generated by $g_{i_0},\ldots, g_{i_N}$ and $H_{\infty}(X)\le B\cdot N(\a)$, then:

Include in $\mathcal L$ all the points in $\o(X)$.
\end{enumerate} 
\end{enumerate}
\item Return the list $\mathcal L$.
\end{enumerate}
\end{alg}

For the algorithm M2, the methods of \S\ref{norm_approach} should be applied to compute $\calC(B)/\mu_K$. Let $\calI(B)$ be a set of generators for all the nonzero principal ideals of $\o_K$ whose norms are in $\calN$. 
For every place $v|\infty$, define real numbers
\begin{align*}
a_v &=\min_{y\in\calI(B)}\Lambda(y)_v, \\
\ell_v &=(n_v/n)(\log B + \log \bN)+D_v-a_v. 
\end{align*}

Let $\calU(B)\subset\R^r$ be the polytope consisting of all points $(t_1,\ldots, t_r)$ that satisfy the inequalities $t_j  \le \ell_{v_j}$ for $1\le j\le r$, and $t_1+\cdots+t_r  \ge -\ell_{v_{r+1}}$. Finally, let $T$ be the $r\times r$ matrix with column vectors $\tilde\Lambda(\epsilon_1),\ldots, \tilde\Lambda(\epsilon_r)$. The results of \S\ref{norm_approach} can be modified in an obvious way to obtain the following result.

\begin{prop}\label{C_B_mod_mu_prop} Let $g_1,\ldots, g_t$ be all the numbers of the form $g=\epsilon_1^{n_1}\cdots\epsilon_r^{n_r}y$
with $y\in \calI(B)$ and $(n_1,\ldots, n_r)\in \Z^r\cap T^{-1}(\calU(B))$ satisfying $\Phi(g)\in\calP(B)$. Then the numbers $0,g_1,\ldots, g_t$ form a complete set of representatives for the orbit space $\calC(B)/\mu_K$.
\end{prop}

From the above proposition we deduce an algorithm for computing $\calC(B)/\mu_K$ that is analogous to Algorithm \ref{C_a_B_alg_second}. We include this algorithm in the following complete description of M2.

\begin{alg}[M2]\label{M2_alg}Computing $\Omega(B)$ using the action of $\mu_K$.\mbox{}
\begin{enumerate}
\item Create an empty list $\mathcal L$. This list will store the points belonging to $\Omega(B)$.
\item Determine ideals $\a_1,\ldots, \a_h$ representing the distinct ideal classes of $\o_K$.
\item Compute an LLL-reduced system $\{\epsilon_1,\ldots,\epsilon_r\}$ of fundamental units in $K$.
\item Construct the set $\calN$.
\item Compute the set $\calI(B)$ by solving norm equations.
\item Compute the numbers $D_v$, $a_v$, and  $\ell_v$ for every place $v|\infty$.
\item Let $T$ be the $r\times r$ matrix with column vectors $\tilde\Lambda(\epsilon_1),\ldots, \tilde\Lambda(\epsilon_r)$.
\item Construct the polytope $T^{-1}(\calU(B))$.
\item Create the set $L=\{0\}$. This set will store representatives for $\calC(B)/\mu_K$.
\item Find all integer lattice points inside $T^{-1}(\calU(B))$.
\item For all such points $(n_1,\ldots,n_r)$, and for every element $y\in \calI(B)$:
\begin{enumerate}
\item Let $g=\epsilon_1^{n_1}\cdots\epsilon_r^{n_r}y$.
\item If $\Phi(g)\in\calP(B)$, then include $g$ in $L$.
\end{enumerate}
\item For each ideal $\a\in\{\a_1,\ldots, \a_h\}$:
\begin{enumerate}
\item Fix an ordering $g_0,g_1,\ldots, g_t$ of the elements of $L\cap\calC(\a, B)$.
\item For every tuple of indices $(i_0,\ldots, i_N)$ such that $0\le i_0\le i_1\le\cdots\le i_N\le t$:
\begin{enumerate}
\item Let $X$ be the point $(g_{i_0},\ldots, g_{i_N})$.
\item If $\a$ is generated by $g_{i_0},\ldots, g_{i_N}$ and $H_{\infty}(X)\le B\cdot N(\a)$, then:

Include in $\mathcal L$ all the points in $\o(X)$.
\end{enumerate} 
\end{enumerate}
\item Return the list $\mathcal L$.
\end{enumerate}
\end{alg}

To conclude our discussion we mention one optional modification that in many cases leads to an improvement in the performance of both M1 and M2. For any point $P=[x_0,\ldots, x_N]\in\P^N(K)$ and any automorphism $\sigma\in\Aut(K/\Q)$, we denote by $P^{\sigma}$ the point
\[P^{\sigma}=[\sigma(x_0),\ldots, \sigma(x_N)]\in\P^N(K).\]
Note that $P$ and $P^{\sigma}$ have the same height and have Galois-conjugate ideal classes. Using this observation one sees that if
$\cc_1,\ldots, \cc_s\in\{\a_1,\ldots, \a_h\}$ are ideals representing the distinct orbits of the action of the group $G=\Aut(K/\Q)$ on the ideal class group $\cl(\o_K)$, then
\[\Omega(B)=\bigcup_{i=1}^sG\cdot\Omega(\cc_i,B).\] 
Thus, in order to compute $\Omega(B)$ it suffices to determine the ideals $\cc_i$, compute the sets $\Omega(\cc_i,B)$, and then let the group $G$ act on these sets. We are not aware of any method for computing the ideals $\cc_1,\ldots, \cc_s$ that does not involve first computing the full list $\a_1,\ldots, \a_h$; hence, this approach carries the additional cost of having to divide the class group into $G$-orbits. However, one finds in practice that if $s$ is substantially smaller than $h$, then the reduction in the number of ideals $\a$ for which the set $\Omega(\a,B)$ must be computed easily makes up for this additional cost. Whether this modification will yield improvements in performance is difficult to determine \textit{a priori}, since there appear to be no results in the literature that would allow a comparison of $s$ and $h$ in terms of standard invariants of $K$. Thus, we suggest that this approach only be taken when the class group computation is not costly, and $B$ is large. In such cases, the additional cost incurred by computing $G$-orbits is minimal, and the savings in time are substantial because every set of the form $\Omega(\a,B)$ would require a significant amount of time to be computed.

\section{The case of unit rank zero}\label{rank0_section}
We discuss here the problem of enumerating all points in the set $\Omega(B)$ when $K$ is a number field with finite unit group (i.e., $K=\Q$ or an imaginary quadratic field). The reasons for treating this case separately are twofold: first, the height function takes a particularly simple shape in this case, reducing the problem to a computation of elements of $\o_K$ with bounded norm; second, for the computation of ideal class groups of imaginary quadratic fields there are specialized algorithms that are faster than the methods that apply to arbitrary number fields \cite{biasse, hafner-mccurley, jacobson}. 

A method for listing all points in $\Omega(B)$ when $K=\Q$ is easily deduced from the following proposition. We will be using here the notation introduced in \S\ref{group_action_section}.

\begin{prop}\label{QQ_prop} Let $\calM$ denote the set of all tuples $(x_0,\ldots, x_N)\in\Z^{N+1}$ such that $0\le x_0\le\cdots\le x_N\le B$ and $\gcd(x_0,\ldots, x_N)=1$. Then 
\[\{P\in\P^N(\Q):H_{\Q}(P)\le B\}=\bigcup_{X\in\calM}\o(X).\]
\end{prop}

\begin{proof} This is a simple exercise using the fact that, for any point $P\in\P^N(\Q)$, we have $H_{\Q}(P)\le B$ if and only if it is possible to write $P=[a_0,\ldots, a_N]$ with $a_i\in\Z\cap[-B,B]$ for all $i$ and $\gcd(a_0,\ldots, a_N)=1$.
\end{proof}

We assume henceforth that $K$ is an imaginary quadratic field. As has been noted earlier, in order to find all points of bounded height in $\P^N(K)$ it suffices to find all points of bounded height with given ideal class. For any nonzero ideal $\a$ of $\o_K$ and real number $B\ge 1$, define
\[\Omega(\a,B)=\{P\in\P^N(K): \cl(P)=\cl(\a)\text{ and } H_K(P)\le B\}.\]
The following proposition reduces the problem of computing $\Omega(\a,B)$ to the computation of the set
\[\calC(\a,B)=\{\gamma\in\a:N_{K/\Q}(\gamma)\le B\cdot N(\a)\}.\]

\begin{prop}\label{increasing_tuples_prop_iq} Let $g_0,g_1,\ldots, g_t\in\calC(\a,B)$ be elements representing all the orbits of the action of $\mu_K$ on $\calC(\a,B)$. Let $\mathcal M$ be the set of all tuples $X=(g_{i_0},\ldots, g_{i_N})$ such that $0\le i_0\le i_1\le\cdots\le i_N\le t$ and $\a$ is the ideal generated by $g_{i_0},\ldots, g_{i_N}$. Then
\[\Omega(\a,B)=\bigcup_{X\in\calM}\o(X).\]
\end{prop}

\begin{proof}
Let $\sigma,\overline\sigma$ be the embeddings $K\into\C$. For any element $\alpha\in K$ we have $N_{K/\Q}(\alpha)=\sigma(\alpha)\overline\sigma(\alpha)=|\sigma(\alpha)|^2$. Hence, for any tuple $(x_0,\ldots, x_N)\in K^{N+1}$ we obtain
\begin{equation}\label{H_infty_iq}
H_{\infty}(x_0,\ldots, x_N)= \max\{N_{K/\Q}(x_0),\ldots, N_{K/\Q}(x_N)\}.
\end{equation}

Suppose that $P\in\Omega(\a,B)$. Since $\cl(P)=\cl(\a)$, there are homogeneous coordinates $[x_0,\ldots, x_N]$ for $P$ such that $\a$ is generated by $x_0,\ldots, x_N$. The condition $H_K(P)\le B$ is then equivalent to $H_{\infty}(x_0,\ldots, x_N)\le B\cdot N(\a)$, so by \eqref{H_infty_iq} we see that $x_i\in\calC(\a,B)$ for every $i$. The remainder of the proof is entirely analogous to the proof of Proposition \ref{increasing_tuples_prop}.
\end{proof}

Let $\a_1,\ldots, \a_h$ be ideals representing the ideal classes of $\o_K$. If we can determine representatives for each orbit space $\calC(\a_i,B)/\mu_K$, then Proposition \ref{increasing_tuples_prop_iq} can be used compute $\Omega(\a_i,B)$ for each $i$, and thus $\Omega(B)$ is obtained. Though it is possible to compute each set $\calC(\a_i,B)$ separately, it would be more efficient to apply some of the ideas introduced in \S\ref{searchspace_section}; in particular, we should compute one set $\calC(B)$ that contains all of the sets $\calC(\a_i,B)$. Define 
\[\calN=\{0\}\cup\bigcup_{i=1}^h\{k\cdot N(\a_i):1\le k\le B\},\]
and let
\[\calC(B)=\{\gamma\in\o_K: N_{K/\Q}(\gamma)\in\calN\},\]
so that $\calC(B)$ contains every set $\calC(\a_i,B)$. Note that the group $\mu_K$ acts on $\calC(B)$. A set $L$ of representatives of the orbits of this action can be found by solving norm equations in $K$, for instance using the algorithm given in \cite{fincke-pohst1}. Having done this, representatives of $\calC(\a_i, B)/\mu_K$ can be obtained for every index $i$ by intersecting $L$ with $\calC(\a_i, B)$. We are thus led to the following algorithm.

\begin{alg}\label{omega_B_alg_iq}Computing $\Omega(B)$ when $K$ is imaginary quadratic.\mbox{}
\begin{enumerate}
\item Create an empty list $\mathcal L$. This list will store the points belonging to $\Omega(B)$.
\item Determine ideals $\a_1,\ldots, \a_h$ representing the ideal classes of $\o_K$.
\item Construct the set $\calN$.
\item Create an empty set $L$. This set will store representatives for $\calC(B)/\mu_K$.
\item For every number $m\in\calN$:

Include in $L$ all elements of $\o_K$ with norm $m$, modulo units.
\item For each ideal $\a\in\{\a_1,\ldots, \a_h\}$:
\begin{enumerate}
\item Fix an ordering $g_0,g_1,\ldots, g_t$ of the elements of the set $L\cap\calC(\a, B)$.
\item For every tuple of indices $(i_0,\ldots, i_N)$ such that $0\le i_0\le i_1\le\cdots\le i_N\le t$:

If $\a$ is generated by $g_{i_0},\ldots, g_{i_N}$, then include in $\mathcal L$ all the points in $\o(X)$.
\end{enumerate}
\item Return the list $\mathcal L$.
\end{enumerate}
\end{alg}

As in the case of number fields with positive unit rank, Algorithm \ref{omega_B_alg_iq} can be improved by using the action of the group $\Gal(K/\Q)$ on the ideal class group of $\o_K$. See the final paragraphs of \S\ref{searchspace_section} for details.

\section{Efficiency of the algorithm}\label{efficiency_section}

In this section we carry out an analysis of the efficiency of our method for computing points of bounded height in $\P^N(K)$. As a measure of the efficiency of the algorithm we consider how many points it generates in the process of searching for points in $\Omega(B)$, and how this quantity compares to the size of $\Omega(B)$. The case $K=\Q$ being trivial in view of Proposition \ref{QQ_prop}, we assume henceforth that $K$ is different from $\Q$. 

Since the problem of computing $\Omega(B)$ is reduced to computing sets $\Omega(\a,B)$ for a finite list of ideals $\a$, we focus first on determining how efficient our method for computing such sets is. For this analysis we use Algorithm \ref{omega_aB_alg}, which is the simplest description of our method. We define the {\it search space} of Algorithm \ref{omega_aB_alg}  to be the set 
\begin{equation}\label{S_aB_def}
\calS(\a,B)=\{[x_0,\ldots, x_N]\in\P^N(K): x_i\in \calC(\a,B) \text{ for all $i$}\}.
\end{equation}
The elements of $\calS(\a,B)$ are all the points that would be considered in the algorithm while searching for points in $\Omega(\a,B)$. Since the size of the search space is largely determined by the size of $\calC(\a,B)$, we begin by understanding the latter. We will bound the size of $\calC(\a,B)$ by the size of the set
\[\widetilde\calC(\a,B)=\{x\in\a: \Phi(x)\in \calP(\a,B)\} \supseteq \calC(\a,B).\]

\begin{prop}\label{ideal_lattice_count} There is a constant $c$, depending only on $K$ and the choice of fundamental units, such that 
\[\#\widetilde\calC(\a,B)=cB+O(B^{1-1/n}).\]
\end{prop}

For the proof of the proposition we will need a classical result concerning lattice points in homogeneously expanding domains; we refer the reader to Lang's book \cite[Chap. VI, \S2]{lang_nt} for details. We say that a subset $T$ of $\R^n$ is $k$-\textit{Lipschitz parametrizable} if there exists a finite number of Lipschitz maps $[0,1]^k\to T$ whose images cover $T$. Recall that a map $f:[0,1]^k\to \R^n$ satisfies the Lipschitz condition if there exists a constant $\alpha$ such that
\[\|f(x)-f(y)\|\le \alpha\cdot\|x-y\|\]
for all $x,y\in [0,1]^k$. 

\begin{lem}\label{lang_lattice_lem} Let $D$ be a bounded subset of $\R^n$ and $L$ a lattice in $\R^n$ with fundamental domain $F$. Assume that the boundary of $D$ is $(n-1)$-Lipschitz parametrizable. Then the number of lattice points in $tD$, for $t\in\R_{>0}$, satisfies
\[\#(L\cap tD)=\frac{\vol D}{\vol F}\cdot t^n + O(t^{n-1}).\]
\end{lem}

\begin{rem} In discussing the above lemma, Lang neglects to show that the condition of $\partial D$ being $(n-1)$-Lipschitz parametrizable implies that $D$ is measurable, so that the volume of $D$ is well defined. Furthermore, Lang does not mention that the $O$ constant depends on the number of maps parametrizing $\partial D$. For a proof of both of these statements, see \cite[Thm. 5.4]{widmer_bdd_height}.
\end{rem}

\begin{proof}[Proof of Proposition \ref{ideal_lattice_count}]

To prove the proposition we apply Lemma \ref{lang_lattice_lem} to the lattice $L=\Phi(\a)$ and the set $D=\calP(\a,1)$. Note that $D$ is bounded and convex, since it is defined as a Cartesian product of closed intervals in $\R$ and closed disks in $\R^2$. It follows from \cite[Thm. 2.6]{widmer_lattice_pts} that the boundary of $D$ is $(n-1)$-Lipschitz parametrizable. Hence, the conditions of Lemma \ref{lang_lattice_lem} are satisfied.

One can see from the definitions that $\calP(\a,B)=B^{1/n}\calP(\a,1)$; thus, applying the lemma we obtain
\begin{equation}\label{C_a_B_size}
\begin{aligned}
\#\widetilde\calC(\a,B) &= \#\left(\Phi(\a)\cap\calP(\a,B)\right)=\#\left(\Phi(\a)\cap B^{1/n}\calP(\a,1)\right) \\
&= \frac{\vol\calP(\a,1)}{\vol F(\a)}\cdot B + O(B^{1-1/n}),
\end{aligned}
\end{equation}
where $F(\a)$ is the fundamental parallelotope defined in \eqref{ideal_fund_par}.
The region $\calP(\a,1)\subset\R^n$ is a Cartesian product of intervals of length $2\cdot N(\a)^{1/n}\exp(D_v)$, where $v$ ranges over the real places of $K$, and disks of radius $N(\a)^{1/n}\exp(D_v/2)$, where $v$ ranges over the complex places. Therefore, 
\begin{align*}
\vol \calP(\a,1)&=\left(\prod_{v \text{ real}}2\cdot N(\a)^{1/n}\exp(D_v)\right)\cdot\left(\prod_{v \text{ complex}}\pi\cdot N(\a)^{2/n}\exp(D_v)\right) \\
&= 2^{r_1}\pi^{r_2}N(\a)\exp(\textstyle\sum_vD_v).
\end{align*}

By \eqref{ideal_volume} we have $\vol F(\a)=2^{-r_2}|\Delta_K|^{1/2}N(\a)$, so the coefficient of $B$ in \eqref{C_a_B_size} is  the constant $c$ given by
\begin{equation}\label{leading_coeff_latt}
c=\frac{2^{r_1}(2\pi)^{r_2}\exp(\textstyle\sum_vD_v)}{|\Delta_K|^{1/2}}.
\end{equation}
This proves the proposition, since $c$ depends only on standard invariants of $K$ and on the numbers $D_v$, which are determined by the choice of fundamental units.
\end{proof}

Having proved asymptotic bounds on the size of $\widetilde\calC(\a,B)$, we can now bound the size of the search space $\calS(\a,B)$.
\begin{prop}\label{class_search_space_size}
Let $m=N+1$ and define a constant $C$ by
\begin{equation}\label{C_const} C=\frac{w\cdot\zeta_K(m)\exp(m\sum_vD_v)}{R\cdot m^r},
\end{equation}
where $w$ is the number of roots of unity in $K$; $R$ is the regulator; $r$ is the rank of the unit group; and $\zeta_K$ is the zeta function of $K$. Then
\[\limsup_{B\to\infty}\frac{\#\calS(\a,B)}{\#\Omega(\a,B)}\le C.\]
\end{prop}

\begin{proof}
We have $\#\calS(\a,B)\le (\#\calC(\a,B))^m\le (\#\widetilde\calC(\a,B))^m$, so Proposition \ref{ideal_lattice_count} implies that
\begin{equation}\label{ideal_search_space_bound}
\#\calS(\a,B)\le c^mB^m+O(B^{m-1/n}).
\end{equation}

Now, Schanuel \cite{schanuel} showed that
\begin{equation}\label{schanuel_class_size}
\#\Omega(\a,B)= B^m\cdot m^r\cdot\frac{R/w}{\zeta_K(m)}\left(\frac{2^{r_1}(2\pi)^{r_2}}{|\Delta_K|^{1/2}}\right)^m+O(B^{m-1/n}).
\end{equation}

(The correct error term is different in the special case where $K=\Q$ and $N=1$, but we are assuming in this section that $K$ is not $\Q$.) The result follows by dividing \eqref{ideal_search_space_bound} and \eqref{schanuel_class_size} and letting $B\to\infty$.
\end{proof}

We can now prove our main result concerning the size of the search space of our method for computing $\Omega(B)$. If $\a_1,\ldots, \a_h$ are ideals representing the distinct ideal classes of $\o_K$, then
\[\Omega(B) = \bigcup_{i=1}^h\Omega(\a_i,B),\]
so $\Omega(B)$ can be obtained by computing $\Omega(\a_i,B)$ for every index $i$. With sets $\calS(\a_i,B)$ defined as in \eqref{S_aB_def}, the search space of this method is the set
\[\calS(B)=\bigcup_{i=1}^h\calS(\a_i,B),\]
whose size we now compare to that of $\Omega(B)$.

\begin{thm}\label{search_space_size} Let $h$ be the class number of $\o_K$ and let $C$ be the constant defined in \eqref{C_const}. Then
\[\limsup_{B\to\infty}\frac{\#\calS(B)}{\#\Omega(B)}\le hC.\]
\end{thm}

\begin{proof}
From the definitions it follows that
\[\frac{\#\calS(B)}{\#\Omega(B)}\le\sum_{i=1}^h\frac{\#\calS(\a_i,B)}{\#\Omega(B)}\le\sum_{i=1}^h\frac{\#\calS(\a_i,B)}{\#\Omega(\a_i,B)}.\]
The theorem is then a consequence of Proposition \ref{class_search_space_size}.
\end{proof}

\section{Sample Computations}\label{computations_section}

The main algorithms of this paper, namely M1 and M2 (Algorithms \ref{M1_alg} and \ref{M2_alg}), have been implemented using \texttt{Sage} \cite{sage}. We give below a series of computations of points of small height in projective spaces using these algorithms. Since our methods apply in particular to the space $\P^1$, and listing points of bounded height in $\P^1$ is equivalent to listing elements of $K$ with bounded height, we also compare our methods to the algorithm developed in \cite{doyle-krumm} for finding elements of small height in number fields. All computations have been done on a MacBook Pro with a 2.7 GHz processor and 16 GB of memory.

\subsection{Comparison of M1 and M2}\label{M1vsM2_section}
The three tables shown below summarize the results of computations of points of bounded height in $\P^N(K)$ for three number fields $K$ and for $N=1,2,3$. In all cases, the height bound $B$ was taken to be 20. The purpose of these examples is to compare the running times of M1 and M2 when applied over number fields of various degrees and projective spaces of various dimensions. The results suggest that M2 is a significantly better method than M1, and indeed we have not found any examples -- among many computations --  where M1 performs better than M2. Theoretically, M1 has an advantage over M2 in that it does not require arithmetic operations with fundamental units; however, this on its own does not seem to make M1 a faster method. The difference in running times is largely due to the fact that M2 generates all points in $\Omega(B)$ starting from the set $\calC(\a,B)/\mu_K$ rather than the larger set $\calC(\a,B)/\{\pm 1\}$.

\begin{table}[!htb]
\centering
\begin{tabular}{|c|c|c|c|}
\hline
$N$ & M1 time & M2 time & Points found\\
\hline
1 &  0.98 s & 0.26 s & 504\\
\hline
2 & 39 s & 10 s & 20,401\\
\hline
3 & 2,465 s & 513 s & 607,344 \\
\hline
\end{tabular}

\medskip
\caption{Computations over the field $K=\Q(\sqrt {17})$}
\end{table}

\begin{table}[!htb]
\centering
\begin{tabular}{|c|c|c|c|}
\hline
$N$ & M1 time & M2 time & Points found\\
\hline
1 & 0.43 s & 0.27 s & 452\\
\hline
2 & 25 s & 14 s & 23,725\\
\hline
3 & 2,707 & 889 s & 888,872 \\
\hline
\end{tabular}

\medskip
\caption{Computations over the field $\Q(\sqrt[3]{2})$}
\end{table}

\begin{table}[!htb]
\centering
\begin{tabular}{|c|c|c|c|}
\hline
$N$ & M1 time & M2 time & Points found\\
\hline
1 & 11 s & 0.41 s & 842\\
\hline
2 & 10,483 s & 51 s & 72,091\\
\hline
3 & - & 10,407 s & 4,926,644 \\
\hline
\end{tabular}

\medskip
\caption{Computations over the field $\Q(\sqrt[4]{-1})$}
\end{table}

\subsection{Elements of bounded height in number fields}\label{P1_section}
 Let $K$ be a number field and let $x\in K$. The height of $x$ is defined to be the number $H_K(x)=H_K(P)$, where $P=[x,1]\in\P^1(K)$. The problem of listing all elements of $K$ up to a given height bound was studied in the article \cite{doyle-krumm}, where an algorithm -- denoted here by DK -- was developed to solve this problem. Since the methods M1 and M2 can be applied to list all elements of bounded height in $\P^1(K)$, it is natural to wonder which algorithm (M1, M2, or DK) is fastest in this context. Given that M1 appears to be generally slower than M2, we will compare only M2 and DK.
 
The four tables shown below list the running times for several computations done with M2 and DK over quadratic, cubic, and quartic number fields, and with several different height bounds $B$. In the tables, the underlying number field $K$ is either given explicitly as an extension of $\Q$, or a defining polynomial for it is given in the variable $x$. We will denote by $g$ a root of the defining polynomial of $K$, so that $K=\Q(g)$.

As indicated by these computations, neither algorithm is always better than the other, at least in their current implementations. There are examples where M2 is significantly faster than DK (see tables 4 and 5), and there are cases where the opposite happens (see tables 6 and 7). This phenomenon appears to be tied to the size of the fundamental units chosen for $K$: when the units are relatively small, DK performs better than M2, but when the units are fairly large, M2 is faster.
The fields $K$ used for the computations in the tables below illustrate the effect of the size of the fundamental units on DK and M2. All of these fields have unit rank 1. Fundamental units for the fields in Tables 4 and 5, where M2 performs best, are given by
\[10771703481902106796084652\cdot g - 1196823028442576899590849641\]
and
\[17597170123512678762361\cdot g^2 + 1494282241689424625747666\cdot g + 7084465262325346055314439,\] 
respectively. In contrast, fundamental units for the fields in Tables 6 and 7, where DK performs best, are given by
\[28g + 295 \hspace{5mm}\text{and}\hspace{5mm} g^3 + 2g^2 + 2g - 2.\]

\begin{table}[!htb]
\centering
\begin{tabular}{|c|c|c|c|}
\hline
$B$ & DK time & M2 time & Elements found\\
\hline
100 & 1.25 s & 0.76 s & 479\\
\hline
1,000 & 92 s & 40 s & 73,007\\
\hline
5,000 & 3,358 s & 1,299 s & 1,826,367 \\
\hline
\end{tabular}

\medskip
\caption{Elements of bounded height in $\Q(\sqrt{12345})$}
\end{table}

\begin{table}[!htb]
\centering
\begin{tabular}{|c|c|c|c|}
\hline
$B$ & DK time & M2 time & Elements found\\
\hline
100 & 0.82 s & 0.51 s & 263\\
\hline
1,000 & 55 s & 19 s & 27,603\\
\hline
5,000 & 1,682 s & 475 s & 731,755 \\
\hline
\end{tabular}

\medskip
\caption{Elements of bounded height in $K:x^3-x+123$}
\end{table} 

\begin{table}[!htb]
\centering
\begin{tabular}{|c|c|c|c|}
\hline
$B$ & DK time & M2 time & Elements found\\
\hline
100 & 1.32 s & 1.97 s & 2,875\\
\hline
1,000 & 49 s & 96 s & 275,615\\
\hline
5,000 & 1,402 s & 3,019 s & 6,795,587 \\
\hline
\end{tabular}

\medskip
\caption{Elements of bounded height in $\Q(\sqrt{111})$}
\end{table} 

\begin{table}[!htb]
\centering
\begin{tabular}{|c|c|c|c|}
\hline
$B$ & DK time & M2 time & Elements found\\
\hline
100 & 0.69 s & 0.82 s & 299\\
\hline
1,000 & 13 s & 23 s & 42,067\\
\hline
5,000 & 202 s & 431 s & 1,092,203 \\
\hline
\end{tabular}

\medskip
\caption{Elements of bounded height in $K: x^4-x+11$}
\end{table}

\subsection*{Acknowledgements} I thank John Doyle and Lenny Fukshansky for help finding references, and the referees for their careful reading of the manuscript and for several helpful comments.
 
\bigskip

\bibliography{mybibliography}{}
\bibliographystyle{amsplain}
\end{document}